\documentclass[12pt]{amsart}
\usepackage{hyperref,cleveref,amssymb}

\theoremstyle{plain}
\newtheorem{theorem}{Theorem}[section]
\newtheorem{proposition}[theorem]{Proposition}
\newtheorem{corollary}[theorem]{Corollary}
\newtheorem{lemma}[theorem]{Lemma}

\theoremstyle{definition}
\newtheorem{remark}[theorem]{Remark}

\newtheorem{example}[theorem]{Example}

\newcommand{\abs}[1]{\lvert#1\rvert}
\newcommand{\norm}[1]{\lVert#1\rVert}

\newcommand{\bignorm}[1]{\bigl\lVert#1\bigr\rVert}

\newcommand{\term}[1]{{\textit{\textbf{#1}}}}   
\renewcommand{\mid}{\::\:}

\newcommand{\goeso}{\xrightarrow{\mathrm{o}}}	
\newcommand{\goesun}{\xrightarrow{\mathrm{un}}} 
\newcommand{\goesunX}{\xrightarrow{\rm{un-}X}} 
\newcommand{\goesunZ}{\xrightarrow{\rm{un-}Z}} 
\newcommand{\goesunXd}{\xrightarrow{\rm{un-}X^\delta}} 
\newcommand{\goesunB}{\xrightarrow{\text{un-$B$}}} 
\newcommand{\goesuo}{\xrightarrow{\mathrm{uo}}}	
\newcommand{\goesnorm}{\xrightarrow{\norm{\cdot}}}	
\newcommand{\goesmu}{\xrightarrow{\mu}}	
\newcommand{\goesae}{\xrightarrow{\mathrm{a.e.}}}	

\DeclareSymbolFont{bbold}{U}{bbold}{m}{n}
\DeclareSymbolFontAlphabet{\mathbbold}{bbold}
\def\one{\mathbbold{1}}

\DeclareMathOperator{\Span}{span}
\DeclareMathOperator{\supp}{supp}

\renewcommand{\le}{\leqslant}
\renewcommand{\ge}{\geqslant}

\begin{document}

\title[Unbounded norm topology]
{Unbounded norm topology\\ beyond normed lattices}

\author{M. Kandi\'c}
\address{Faculty of Mathematics and Physics, University of Ljubljana,
  Jadranska 19, 1000 Ljubljana, Slovenia}
\email{marko.kandic@fmf.uni-lj.si}

\author{H. Li} \address{School of Mathematics, Southwest Jiaotong
  University, Chengdu, Sichuan, 610000, China.}
\email{lihuiqc168@126.com}

\author{V.G. Troitsky}
\address{Department of Mathematical and Statistical Sciences,
         University of Alberta, Edmonton, AB, T6G\,2G1, Canada.}
\email{troitsky@ualberta.ca}

\thanks{The first author acknowledges the financial support from the
  Slovenian Research Agency (research core funding No. P1-0222). The
  third author was supported by an NSERC grant.}  

\date{\today}

\begin{abstract}
  In this paper, we generalize the concept of unbounded norm (un)
  convergence: let $X$ be a normed lattice and $Y$ a vector lattice
  such that $X$ is an order dense ideal in $Y$; we say that a net
  $(y_\alpha)$ un-converges to $y$ in $Y$ with respect to $X$ if
  $\bignorm{\abs{y_\alpha-y}\wedge x}\to 0$ for every $x\in X_+$. We
  extend several known results about un-convergence and un-topology to
  this new setting. We consider the special case when $Y$ is the
  universal completion of $X$. If $Y=L_0(\mu)$, the space of all
  $\mu$-measurable functions, and $X$ is an order continuous Banach
  function space in $Y$, then the un-convergence on $Y$ agrees with
  the convergence in measure. If $X$ is atomic and order complete and
  $Y=\mathbb R^A$ then the un-convergence on $Y$ agrees with the
  coordinate-wise convergence.
\end{abstract}

\maketitle

\section{Introduction and preliminaries}

All vector lattices in this paper are assumed to be Archimedean.
A net $(x_\alpha)$ in a normed lattice $X$ is said to \term{un-converge} to
$x$ if $\abs{x_\alpha-x}\wedge u\to 0$ in norm for every $u\in
X_+$. This convergence, as well as the corresponding
\term{un-topology}, has been introduced and studied in
\cite{Troitsky:04,DOT,KMT}. In particular, if $X=L_p(\mu)$ for a
finite measure $\mu$ and $1\le p<\infty$, then un-convergence agrees
with convergence in measure. Convergence in measure naturally extends
to $L_0(\mu)$, the space of all measurable functions (as usual, we
identify functions that are equal a.e.). However, $L_0(\mu)$ is not a
normed lattice, so that the preceding definition of un-convergence does
not apply to $L_0(\mu)$. Motivated by this example, we generalize
un-convergence as follows. 

\emph{Throughout the paper, $X$ is a normed lattice and $Y$ is a
  vector lattice such that $X$ is an ideal in~$Y$.} For a net
$(y_\alpha)$ in~$Y$, we say that $y_\alpha$ \term{un-converges to}
$y\in Y$ \term{with respect to} $X$ if
$\bignorm{\abs{y_\alpha-y}\wedge x}\to 0$ for every $x\in X_+$; we
write $y_\alpha\goesunX y$. In the case when $(y_\alpha)$ and $y$ are
in~$X$, it is easy to see that $y_\alpha\goesunX y$ iff
$y_\alpha\goesun y$ in~$X$. Therefore, the new convergence is an
extension of the un-convergence on~$X$. We will write
$y_\alpha\goesun y$ instead of $y_\alpha\goesunX y$ when there is no
confusion.

In this paper, we study properties of extended un-convergence and
un-topology. We show that many properties of the original
un-convergence remain valid in this setting. We study when this
convergence is independent of the choice of~$X$. We consider several
special cases: when $Y$ is the universal completion of~$X$, when $X$
is atomic and $Y=\mathbb R^A$, and when $Y=L_0(\mu)$ and $X$ is a
Banach function space in $L_0(\mu)$.

\begin{example}\label{Lp}
  Let $X=L_p(\mu)$ where $\mu$ is a finite measure
  and $1\le p<\infty$, and $Y=L_0(\mu)$. A net in $L_0(\mu)$
  un-converges to zero with respect to $L_p(\mu)$ iff it converges to
  zero in measure. The proof is analogous to Example~23
  of~\cite{Troitsky:04}.  In particular, $L_p(\mu)$ spaces for all $p$
  in $[1,\infty)$ induce the same un-convergence on $L_0(\mu)$; namely,
  the convergence in measure.
\end{example}

Just as for the original un-convergence, we have $y_\alpha\goesun y$
in $Y$ iff $\abs{y_\alpha-y}\goesun 0$. This often allows one to
reduce general un-convergence to the un-convergence of positive nets
to zero. The proof of the following fact is similar to that of
\cite[Lemma~2.1]{DOT}. 

\begin{proposition}\label{ops}
  Un-convergence in $Y$ preserves algebraic and lattice operations. 
\end{proposition}

In Section~7 of~\cite{DOT}, it was observed that in the case when
$X=Y$ the un-convergence on a normed lattice is given by a topology,
and a base of zero neighbourhoods for that topology was given. We will
now proceed analogously. Given a positive real number $\varepsilon$
and a positive vector $x\in X$, define
\begin{displaymath}
  U_{\varepsilon,x}
  =\Bigl\{y\in Y\mid \bignorm{\abs{y}\wedge x}<\varepsilon\Bigr\}.
\end{displaymath}
Similarly to  Section~7 of~\cite{DOT}, one may verify that
\begin{itemize}
  \item zero is contained in  $U_{\varepsilon,x}$ for all
    $\varepsilon>0$ and $x\in X_+$;
  \item For every $\varepsilon_1,\varepsilon_2>0$ and every
    $x_1,x_2\in X_+$ there exists $\varepsilon>0$ and $x\in X_+$ such
    that $U_{\varepsilon,x}\subseteq U_{\varepsilon_1,x_1}\cap
    U_{\varepsilon_2,x_2}$;
  \item Given $y\in U_{\varepsilon,x}$ for some $y\in Y$, $x\in X_+$,
    and $\varepsilon>0$, we have $y+U_{\delta,x}\subseteq
    U_{\varepsilon,x}$ for some $\delta>0$.
\end{itemize}
For every $y\in Y$, define the family $\mathcal N_y$ of subsets of $Y$
as follows: $W\in\mathcal N_y$ if $y+U_{\varepsilon,x}\subseteq W$ for some
$\varepsilon>0$ and $x\in X_+$. It follows from, e.g., Theorem~3.1.10
of \cite{Runde:05} that there is a unique topology on $Y$ such that
$\mathcal N_y$ is exactly the set of all neighbourhoods of $y$ for
every $y\in Y$. It is also easy to see that $y_\alpha\goesun y$ in $Y$
iff for every $\varepsilon>0$ and every $x\in X_+$, the set
$U_{\varepsilon,x}$ contains a tail of the net $(y_\alpha-y)$; it
follows that the un-convergence on $Y$ with respect to $X$ is exactly
the convergence with respect to this topology. We call it the
\term{un-topology on $Y$ induced by $X$}.

There are, however, two important differences with~\cite{DOT}.
First, unlike in~\cite{DOT}, this topology need not be Hausdorff. 

\begin{example}\label{Lp-band}
  Let $Y=L_p(\mu)$, where $\mu$ is a finite measure and
  $1\le p<\infty$; let $X$ be a band in~$Y$, i.e., $X=L_p(A,\mu)$,
  where $A$ is a measurable set. Consider the un-convergence on $Y$
  with respect to~$X$. In this case, for every net $(y_\alpha)$
  in the disjoint complement $X^d$ of $X$ in $Y$ and every $y\in X^d$
  we have $y_\alpha\goesunX y$. This shows that un-limits need not be
  unique, so that un-topology need not be Hausdorff. Note, also, that
  the un-convergence on $Y$ induced by $X$ is different from the
  ``native'' un-convergence of~$Y$.
\end{example}

Recall that a sublattice $F$ of a vector lattice $E$ is
\begin{itemize}
\item \term{order dense} if for every non-zero $x\in E_+$ there exists
  $y\in F$ such that $0<y\le x$; and
\item \term{majorizing} if for every $x\in E_+$ there exists $y\in F$
  with $x\le y$.
\end{itemize}
An ideal $F$ of $E$ is order dense iff $u=\sup\bigl\{u\wedge v\mid
v\in F_+\bigr\}$ for every $u\in E_+$; see
\cite[Theorem~1.27]{Aliprantis:03}. 

\begin{proposition}\label{Hausdorff}
  The un-topology on $Y$ induced by $X$ is Hausdorff iff $X$ is order
  dense in~$Y$.
\end{proposition}

\begin{proof}
  Suppose that the topology is Hausdorff and let $0<y\in Y$. Then $y\notin
  U_{\varepsilon,x}$ for some $\varepsilon>0$ and some $x\in X_+$. It
  follows that $y\wedge x\ne 0$. Note that $y\wedge x\in
  X$ and $y\wedge x\le y$. Therefore, $X$ is order dense.

  Conversely, suppose that $X$ is order dense and take $0\ne y\in
  Y$. Find $x\in X$ with $0<x\le\abs{y}$. Then $y\notin
  U_{\varepsilon,x}$ where $\varepsilon=\norm{x}$.
\end{proof}

The second important difference between our setting and that
of~\cite{DOT} is linearity. It is shown in~\cite{DOT} that the base
zero neighbourhoods $V_{u,\varepsilon}$ are absorbing; then
Theorem~5.1 of~\cite{Kelley:76} is used to conclude that the resulting
topology is linear. In our setting, however, the sets
$U_{\varepsilon,x}$ need not be absorbing and the topology need not be
linear.

\begin{example}
  Let $X=\ell_\infty$ and $Y=\mathbb R^{\mathbb N}$. Put $x=\one$, the
  constant one sequence, $\varepsilon=1$, and $z=(1,2,3,\dots)$. Then
  $U_{\varepsilon,x}=\{y\in Y\mid \sup\abs{y_i}<1\}$. It is easy to
  see that no scalar multiple of $z$ is in $U_{\varepsilon,x}$, hence
  $U_{\varepsilon,x}$ is not absorbing. It also follows that the
  sequence $\frac1n z$ does not un-converge to zero as
  $n\to\infty$. This shows that the un-topology on $Y$ induced by $X$
  is not linear.
\end{example}

However, it is easy to see that the un-topology on $Y$ induced by $X$
is translation invariant. Moreover, addition is jointly continuous by
Proposition~\ref{ops}, so the problem is only with the continuity of
scalar multiplication. It is easy to see that the un-topology on $Y$
generated by $X$ is linear when $X$ is order continuous or when $Y$ is
a normed lattice and $\norm{\cdot}_X$ and $\norm{\cdot}_Y$ agree on
$X$. Note also that, in general, the restriction of this topology to
$X$ agrees with the ``native'' un-topology of $X$, which is Hausdorff
and linear.

\bigskip

We finish the introduction with the following two easy facts that will
be used throughout the rest of the paper. The following is an analogue
of \cite[Lemma~1.2]{KMT}.

\begin{proposition}\label{un-monot}
  Suppose that $X$ is order dense in~$Y$. If $y_\alpha\uparrow$ and
  $y_\alpha\goesun y$ in $Y$ then $y_\alpha\uparrow y$.
\end{proposition}

\begin{proof}
  Without loss of generality, $y_\alpha\ge 0$ for each $\alpha$;
  otherwise, pass to a tail $(y_\alpha)_{\alpha\ge\alpha_0}$ and
  consider the net
  $(y_\alpha-y_{\alpha_0})_{\alpha\ge\alpha_0}$. Therefore, we may
  assume that $y\ge 0$ by Proposition~\ref{ops}. We claim that
  $y_\alpha\wedge x\uparrow y\wedge x$ for every $x\in X_+$. Indeed,
  it follows from $\abs{y_\alpha-y}\wedge x\goesnorm 0$ that
  $y_\alpha\wedge x\goesnorm y\wedge x$. Since the net
  $(y_\alpha\wedge x)$ is increasing, it follows that
  $y_\alpha\wedge x\uparrow y\wedge x$. This proves the claim.

  Since $X$ is order dense in~$Y$, we have
  \begin{displaymath}
    y=\sup\limits_{x\in X_+}y\wedge x
    =\sup\limits_{x\in X_+}\sup\limits_\alpha y_\alpha\wedge x
    =\sup\limits_\alpha\sup\limits_{x\in X_+} y_\alpha\wedge x
    =\sup\limits_\alpha y_\alpha.
  \end{displaymath}
\end{proof}

For a vector lattice~$E$, we write $E^\delta$ for the order (Dedekind)
completion of $E$. Recall that $E$ is order dense and majorizing
in~$E^\delta$; moreover, these properties
characterize~$E^\delta$. Suppose that $F$ is an ideal in $E$.  Let $Z$
be the ideal generated by $F$ in~$E^\delta$. It is easy to see that
$F$ is order dense and majorizing in~$Z$; it follows that we may
identify $Z$ with~$F^\delta$. Therefore, if $F$ is an ideal of $E$
then $F^\delta$ may be viewed as an ideal in~$E^\delta$.

In particular, we view $X^\delta$ as an ideal in~$Y^\delta$. The norm
on $X$ admits an extension to a lattice norm on $X^\delta$; see, e.g.,
\cite[p. 179]{Vulikh:67} or \cite[p. 26]{Abramovich:02}. Note that
such an extension need not be unique; see \cite {Solovev:66}. The following
proposition is valid for any such extension; the proof is
straightforward.

\begin{proposition}\label{deltas}
  Let  $X$ be a normed lattice which is an ideal in a vector
  lattice~$Y$, and $(y_\alpha)$ a net in~$Y$. Then $y_\alpha\goesunX 0$
  in $Y$ iff $y_\alpha\goesunXd 0$ in~$Y^\delta$.
\end{proposition}

\section{Uniqueness of un-topology}

Example~\ref{Lp-band} shows that un-convergence on $Y$ may depend on
the choice of~$X$. Here is another example of the same phenomenon.

\begin {example}
  Let $Y=C[0,1]$ equipped with the supremum norm and let $X$ be the
  subspace of $Y$ consisting of all the functions which vanish
  at~$0$. Then $X$ is an order dense ideal in $Y$ which is not norm
  dense. Let $f_n\in Y$ be such that $\norm{f_n}=1$ and
  \begin{math}
    \supp f_n=\bigl[\frac{1}{n+1},\frac{1}{n}\bigr].
  \end{math}
  Then $(f_n)$ un-converges to zero in $Y$ with respect to~$X$, but fails
  to un-converge in $Y$ (with respect to $Y$). Therefore, the
  un-toplogy in $Y$ induced by $X$ does not agree with the
  ``native'' un-toplogy in~$Y$.
\end {example}

In the rest of this section, we consider situations when different
normed ideals of $Y$ induce the same un-topology on~$Y$.

\begin{proposition}\label{dense}
  Let $Z$ be a norm dense ideal in~$X$. Then $Z$ and $X$ induce the
  same un-topology on~$Y$.
\end{proposition}

\begin{proof}
  It suffices to show that $y_\alpha\goesunX 0$ iff
  $y_\alpha\goesunZ 0$ for every net $(y_\alpha)$ in $Y_+$. It is
  clear that $y_\alpha\goesunX 0$ implies $y_\alpha\goesunZ 0$. To
  prove the converse, suppose that $y_\alpha\goesunZ 0$; fix
  $x\in X_+$ and $\varepsilon>0$. Find $z\in Z_+$ such that
  $\norm{x-z}<\varepsilon$. By assumption, $y_\alpha\wedge z\to
  0$. This implies that there exists $\alpha_0$ such that
  $\norm{y_\alpha\wedge z}<\varepsilon$ whenever
  $\alpha\ge\alpha_0$. It follows that
  \begin{displaymath}
    \norm{y_\alpha\wedge x}
    \le\norm{y_\alpha\wedge z}+\norm{x-z}
    <2\varepsilon,
  \end{displaymath}
  so that $y_\alpha\goesunX 0$.
\end{proof}

\begin{example}
  Let $X=c_0$ and $Y=\ell_\infty$. Since $Y$ has a strong unit, the
  ``native'' un-topology on $Y$ agrees with its norm topology. We
  claim that the un-convergence induced on $Y$ by $X$ is the
  coordinate-wise convergence. Indeed, if $y_\alpha\goesunX 0$ in $Y$
  then $\abs{y_\alpha}\wedge e_i\to 0$ for every~$i$, where $e_i$ is
  the $i$-th unit vector in $c_0$; it follows that $y_\alpha$
  converges to zero coordinate-wise. Conversely, if $y_\alpha$
  converges to zero coordinate-wise in $Y$ then $\abs{y_\alpha}\wedge
  x\to 0$ for every $x\in c_{00}$, so that
  \begin{math}
    y_\alpha\xrightarrow{\text{un-$c_{00}$}}0.
  \end{math}
  Proposition~\ref{dense} now yields that
  \begin{math}
    y_\alpha\xrightarrow{\text{un-$c_{0}$}}0.
  \end{math}
\end{example}

In Proposition~\ref{dense}, the norm on $Z$ was the restriction of the
norm of~$X$. We would now like to consider situations where $X$ and
$Z$ have different norms, e.g., $X=L_p(\mu)$ and $Z=L_q(\mu)$, where
$\mu$ is a finite measure and $1\le p\le q<\infty$. We need the
following lemma, which is a variant of Amemiya's Theorem (see, e.g.,
Theorem~2.4.8 in~\cite{Meyer-Nieberg:91}). We provide a proof for
completeness.

\begin{lemma}\label{Amemiya}
  Let $X$ be a Banach lattice and $Z$ an order continuous normed
  lattice such that $Z$ continuously embeds into $X$ as an
  ideal. Then the norm topologies of $X$ and $Z$ agree on order
  intervals of~$Z$.
\end{lemma}

\begin{proof}
  Fix $u\in Z_+$ and let $(z_n)$ be a sequence in $[-u,u]$. Since the
  embedding of $Z$ into $X$ is continuous, if $z_n\to z$ in $Z$ then
  $z_n\to z$ in~$X$. Suppose now that $z_n\to z$ in~$X$. Then
  $z\in[-u,u]$, hence $z\in Z$. Without loss
  of generality, $z=0$. For the sake of contradiction, suppose that
  $(z_n)$ does not converge to zero in~$Z$. Passing to a subsequence,
  we can find $\varepsilon>0$ such that $\norm{z_n}_Z>\varepsilon$ for
  every~$n$. Since $\norm{z_n}_X\to 0$, passing to a further
  subsequence we may assume that $z_n\goeso 0$ in~$X$. Since
  $z_n\in[-u,u]$ and $Z$ is an ideal in~$X$, we conclude that
  $z_n\goeso 0$ in~$Z$. Since $Z$ is order continuous, this yields
  $\norm{z_n}_Z\to 0$; a contradiction.
\end{proof}

\begin{theorem}\label{nested}
  Suppose that $X$ is a Banach lattice and $Z$ is an order continuous
  normed lattice such that $Z$ continuously embeds into $X$ as a norm
  dense ideal. Then $X$ and $Z$ induce the same un-topology on~$Y$.
\end{theorem}

\begin{proof}
  Again, it suffices to show that $y_\alpha\goesunX 0$ iff
  $y_\alpha\goesunZ 0$ for every net $(y_\alpha)$ in~$Y_+$.
  Suppose that $y_\alpha\goesunX 0$. Fix $z\in Z_+$. Then
  $\norm{y_\alpha\wedge z}_X\to 0$. Since this net is contained in
  $[0,z]$, Lemma~\ref{Amemiya} yields that $\norm{y_\alpha\wedge
    z}_Z\to 0$. Thus, $y_\alpha\goesunZ 0$.

  Suppose now that $y_\alpha\goesunZ 0$. Since the inclusion of $Z$
  into $X$ is continuous, we have $\norm{y_\alpha\wedge z}_X\to 0$ for
  every $z\in Z_+$. It follows that $(y_\alpha)$ un-converges to zero
  with respect to $\bigl(Z,\norm{\cdot}_X\bigr)$. It follows now from
  Proposition~\ref{dense} that $y_\alpha\goesunX 0$.
\end{proof}

\begin{theorem}\label{un-same}
  Let $\bigl(X_1,\norm{\cdot}_1\bigr)$ and
  $\bigl(X_2,\norm{\cdot}_2\bigr)$ be two order continuous Banach
  lattices which are both order dense ideals in a vector
  lattice~$Y$. Then $X_1$ and $X_2$ induce the same un-topology
  on~$Y$.
\end{theorem}

\begin{proof}
  Let $Z=X_1\cap X_2$. It is easy to see that $Z$ is an order dense
  ideal of~$Y$. In particular, $Z$ is an order (and, therefore, norm)
  dense ideal in both $X_1$ and~$X_2$. For $z\in Z$, define
  $\norm{z}=\max\bigl\{\norm{z}_{X_1},\norm{z}_{X_2}\bigr\}$. Then $Z$
  is an order continuous normed lattice and the inclusions of $Z$ into
  $X_1$ and $X_2$ are continuous. Applying Theorem~\ref{nested} to
  pairs $(X_1,Z)$ and $(X_2,Z)$, we get the desired result.
\end{proof}

\section{Un-topology and weak units}

In this section, we assume that $X$ is a Banach lattice, though most
of the results extend to the case when $X$ is only a normed
lattice. As before, we assume that $X$ is also an ideal in a vector
lattice~$Y$. It was shown in Lemma~2.11 of~\cite{DOT} that a net
$(x_\alpha)$ in a Banach lattice with a quasi-interior point $u$
un-converges to $x$ iff $\abs{x_\alpha-x}\wedge u\to 0$. We now extend
this result.

\begin{proposition}\label{qip}
  Suppose that $u\ge 0$ is a quasi-interior point of $X$ and
  $(y_\alpha)$ and $y$ are in~$Y$. Then $y_\alpha$ un-converges to $y$
  with respect to $X$ iff $\abs{y_\alpha-y}\wedge u\goesnorm 0$ in~$X$.
\end{proposition}

\begin{proof}
  The forward implication is trivial. Suppose that
  $\abs{y_\alpha-y}\wedge u\goesnorm 0$ in $X$. Then, clearly,
  $\abs{y_\alpha-y}\wedge x\goesnorm 0$ for every positive $x$ in the
  principal ideal~$I_u$. Now apply Proposition~\ref{dense}.
\end{proof}

\begin{corollary}
  Suppose that $X$ has a quasi-interior point and $y_\alpha\goesun 0$
  in~$Y$. There exist $\alpha_1<\alpha_2<\dots$ such that
  $y_{\alpha_n}\goesun 0$.
\end{corollary}

It was shown in Theorem~3.2 of~\cite{KMT} that un-topology on $X$ is
metrizable iff $X$ has a quasi-interior point. We now extend this
result to~$Y$.

\begin{theorem}\label{metriz}
  The following are equivalent:
  \begin{enumerate}
  \item\label{metriz-Y} The un-topology on $Y$ is metrizable;
  \item\label{metriz-X} The un-topology on $X$ is metrizable and $X$ is
    order dense in~$Y$;
  \item\label{metriz-qip} $X$ contains a quasi-interior point which is
    also a weak unit in~$Y$.
  \end{enumerate}
\end{theorem}

\begin{proof}
  \eqref{metriz-Y}$\Rightarrow$\eqref{metriz-X} Suppose that the
  un-topology on $Y$ is metrizable. It follows immediately that its
  restriction to $X$ is metrizable. Furthermore, being metrizable, the
  un-topology on $Y$ is Hausdorff, so that $X$ is order dense in $Y$
  by Proposition~\ref{Hausdorff}.

  \eqref{metriz-X}$\Rightarrow$\eqref{metriz-qip} By Theorem~3.2
  in~\cite{KMT}, $X$ has a quasi-interior point, say~$u$. It follows
  that $u$ is a weak unit in~$X$. Since $X$ is order dense in~$Y$, $u$
  is a weak unit in~$Y$.

  \eqref{metriz-qip}$\Rightarrow$\eqref{metriz-Y} Let $u$ be as
  in~\eqref{metriz-qip}. For $y_1,y_2\in Y$, define 
  \begin{math}
    d(y_1,y_2)=\bignorm{\abs{y_1-y_2}\wedge u}.
  \end{math}
  Since $u$ is a weak unit in~$Y$, $d$ is a metric on $Y$; the proof
  is similar to that of \cite[Theorem~3.2]{KMT}. Note that
  $d(y_\alpha,y)\to 0$ iff $\abs{y_\alpha-y}\wedge u\goesnorm 0$ in~$X$. By
  Proposition~\ref{qip}, this is equivalent to $y_\alpha\goesun
  y$. Therefore, $d$ is a metric for un-topology.
\end{proof}


\section{Atomic Banach lattices}
\label{sec:atomic}

Recall that a positive non-zero vector $a$ in a vector lattice $E$ is
an \term{atom} if the principal ideal $I_a$ equals the span of~$a$. In
this case, $I_a$ is a projection band. The corresponding band
projection $P_a$ has the form $P_ax=\varphi_a(x)a$, where $\varphi$ is
the \term{coordinate functional} of~$a$. We say that $E$ is atomic if
it equals the band generated by all the atoms of~$E$.

Suppose that $E$ is atomic. Let $A$ be a maximal
collection of pair-wise disjoint atoms in~$E$. A net $(x_\alpha)$ in
$E$ converges to zero \term{coordinate-wise} if
$\varphi_a(x_\alpha)\to 0$ for every atom $a$ (or, equivalently, for
every $a\in A$). For every $x\in E_+$, one has
\begin{math}
  x=\sup\bigl\{\varphi_a(x)a\mid a\in A\bigr\}.
\end{math}
This allows one to identify $x$ with the function
$a\in A\mapsto\varphi_a(x)$ in $\mathbb R^A$. Extending this map
to~$E$, one produces a lattice isomorphism from $E$ onto an order
dense sublattice of~$\mathbb R^A$. Thus, every atomic vector lattice
can be identified with an order dense sublattice of~$\mathbb
R^A$. Furthermore, $E$ is order complete iff it is an ideal in~$\mathbb R^A$. For details, we refer the reader to
\cite[p.~143]{Schaefer:74}.  With a minor abuse of notation, we
identify every $x\in E$ with the function $a\mapsto\varphi_a(x)$ in
$\mathbb R^A$; in particular, we identify $a\in A$ with the
characteristic function of~$\{a\}$.

It was shown in Corollary~4.14 of \cite{KMT} that if $X$ is an atomic
order continuous Banach lattice then un-convergence in $X$ coincides
with coordinate-wise convergence. Taking $X=\ell_\infty$ shows that
this may fail when $X$ is not order continuous.

Let $X$ be an order complete atomic Banach lattice, represented as an
order dense ideal in~$\mathbb R^A$. The coordinate-wise convergence on
$X$ is then the restriction of the point-wise convergence
on~$\mathbb R^A$.  We can now define un-convergence on $\mathbb R^A$
induced by~$X$.

\begin{proposition}\label{atom-uo-ptwise}
  Let $X$ be an order complete atomic Banach lattice represented as an
  order dense ideal in~$\mathbb R^A$. For a net $(y_\alpha)$
  in~$\mathbb R^A$, if $y_\alpha\goesun 0$ then $y_\alpha\to 0$
  point-wise. The converse is true iff $X$ is order continuous.
\end{proposition}

\begin{proof}
  Suppose that $y_\alpha\goesun 0$ in~$\mathbb R^A$. For every $a\in
  A$, we have $\abs{y_\alpha}\wedge a\goesnorm 0$; it follows
  easily that $y_\alpha(a)\to 0$.

  Suppose that $X$ is order continuous and $(y_\alpha)$ is a net in
  $\mathbb R^A$ which converges to zero point-wise.  Let $Z=\Span A$
  in~$X$. Clearly, $Z$ is an ideal in $X$; it is norm dense because
  $X$ is order continuous.  It is easy to see that
  $\abs{y_\alpha}\wedge z\to 0$ in norm whenever $z\in Z_+$.
  Therefore, $y_\alpha\goesunZ 0$. By Proposition~\ref{dense}, we have
  $y_\alpha\goesunX 0$.

  Suppose now that if $y_\alpha\to 0$ point-wise then
  $y_\alpha\goesun 0$ for every net $(y_\alpha)$ in $\mathbb R^A$. To
  prove that $X$ is order continuous, it suffices to show that every
  disjoint  order bounded sequence $(x_n)$ in $X$ is norm
  null. Clearly, such a sequence converges to zero coordinate-wise,
  hence, by assumption, $x_n\goesun 0$. Since $(x_n)$ is order
  bounded, we have $x_n\goesnorm 0$.
\end{proof}

What happens when $X$ is order complete but not order continuous?
Recall that in this case, $X$ contains a lattice copy of
$\ell_\infty$; see, e.g.,
\cite[Corollary~2.4.3]{Meyer-Nieberg:91}. So the following example is,
in some sense, representative.

\begin{example}\label{linfty}
  Let $X=\ell_\infty(\Omega)$ for some set~$\Omega$. In this case, $X$
  is atomic and may be viewed as an order dense ideal of~$\mathbb
  R^\Omega$. Note that $u=\one$ is a strong unit and, therefore, a
  quasi-interior point in~$X$. It can now be easily deduced from
  Proposition~\ref{qip} that the un-convergence induced on  $\mathbb
  R^\Omega$ by $X$ coincides with uniform convergence. 
\end{example}

\begin{example}
  Both $\ell_1$ and $\ell_\infty$ may be viewed as order dense ideals
  in~$\mathbb R^{\mathbb N}$. Let $(e_n)$ be the standard unit
  vector sequence. Then $e_n$ un-converges to zero in $\mathbb R^{\mathbb N}$
  with respect to $\ell_1$ but not to~$\ell_\infty$.
\end{example}

\section{Banach function spaces}
\label{sec:BFS}

We are going to extend Example~\ref{Lp} to Banach function
spaces. Throughout this section, we say that $X$ is a \term{Banach
  function space} if it is a Banach lattice which is an order dense
ideal in $L_0(\mu)$ for some measure space $(\Omega,\mathcal F,\mu)$.
Throughout this section, we assume that $\mu$ is $\sigma$-finite.
Note that by \cite[Corollary~5.22]{Abramovich:02}, $X$ has a weak
unit. Theorem~\ref{metriz} yields the following.

\begin{corollary}\label{bfs-metriz}
  Let $X$ be an order continuous Banach function space over
  $(\Omega,\mathcal F,\mu)$. The un-topology induced by $X$ on
  $L_0(\mu)$ is metrizable.
\end{corollary}

Next, we are going to show the un-convergence induced by $X$ on $L_0(\mu)$ 
agrees with ``local'' convergence in measure. 
Let $A\in\mathcal F$. For $x\in L_0(\mu)$ we write $x_{|A}$ for the
restriction of $x$ to $A$. With some abuse of notation, we identify it
with $x\cdot\chi_A$. For $p\in[0,\infty]$, we write $L_p(A)$ for the
set $\bigl\{x_{|A}\mid x\in L_p(\mu)\bigr\}$.

\begin{theorem}\label{bfs-measure}
  Let $X$ be an order continuous Banach function space over
  $(\Omega,\mathcal F,\mu)$, and $(y_\alpha)$ is a net in
  $L_0(\mu)$. Then $y_\alpha\goesunX 0$ iff ${y_\alpha}_{|A}$
  converges to zero in measure whenever $\mu(A)<\infty$.
\end{theorem}

\begin{proof}
  By Theorem~\ref{un-same}, we may assume without loss of generality
  that $X=L_1(\mu)$. Suppose that
  $y_\alpha\xrightarrow{\text{un-}L_1(\mu)}0$, and let $A$ be a
  measurable set of finite measure. For every positive $x\in L_1(A)$,
  we have $\abs{y_\alpha}\wedge x\to 0$ in norm. It follows that the
  net $({y_\alpha}_{|A})$ in $L_0(A)$ is un-null with respect to
  $L_1(A)$. By Example~\ref{Lp}, we conclude that $({y_\alpha}_{|A})$
  converges to zero in measure.

  Conversely, suppose that ${y_\alpha}_{|A}$ converges to zero in
  measure whenever $\mu(A)<\infty$. Let $Z$ be the set of all
  functions in $L_1(\mu)$ which vanish outside of a set of finite
  measure. It is easy to see that $Z$ is a norm dense ideal in
  $L_1(\mu)$. Fix $z\in Z_+$. Find $A\in\mathcal F$ such that
  $\mu(A)<\infty$ and $z$ vanishes outside of~$A$. By assumption,
  ${y_\alpha}_{|A}$ converges to zero in measure. It follows by
  Example~\ref{Lp} that ${y_\alpha}_{|A}$ un-converges in $L_0(A)$
  with respect to $L_1(A)$. In particular, we have
  $\abs{y_\alpha}\wedge z\to 0$ in $L_1(A)$. It follows that
  $y_\alpha\goesunZ 0$ in $L_0(\mu)$. Proposition~\ref{dense} yields
  $y_\alpha\xrightarrow{\text{un-}L_1(\mu)}0$.
\end{proof}

\section{Universal completion}

In this section, we consider the special case when $Y=X^u$.  Recall
that every vector lattice $E$ may be identified with an order dense
sublattice of its \emph{universal completion}~$E^u$, see, e.g.,
\cite[Theorem 7.21]{Aliprantis:03}. If, in addition, $E$ is order
complete, then $E$ is an ideal in $E^u$ by, e.g.,
\cite[Theorem~1.40]{Aliprantis:03}. It follows from
\cite[Theorem~7.23]{Aliprantis:03} that if $F$ is an order dense
sublattice of $E$ then $F^u=E^u$.

Suppose that $X$ is an order complete Banach lattice. The norm need
not extend to $X^u$; the latter is only a vector lattice. However,
using our construction, the un-topology of $X$ admits an extension
to~$X^u$. Combining results of the preceding sections, we get the
following.

\begin{theorem}\label{univ-metriz}
  Let $X$ be an order complete Banach lattice. Un-topology on
  $X^u$ is Hausdorff; it is metrizable iff $X$ has a quasi-interior
  point.
\end{theorem}

The following examples underline that much of what we did in
previous section fits into this general framework.

\begin{example}\label{Xu-atom}
  In Section~\ref{sec:atomic}, we discussed the un-convergence induced
  by an order complete atomic Banach lattice $X$ onto~$\mathbb R^A$,
  where $A$ is a maximal pair-wise disjoint collection of atoms. It is
  easy to see that, in this case, $X^u$ may be identified with~$\mathbb R^A$.
\end{example}

\begin{example}
  In Example~\ref{Lp}, we discussed the case where $X=L_p(\mu)$ and
  $Y=L_0(\mu)$, where $\mu$ is a finite measure and $0\le p<\infty$.
  By \cite[Theorem 7.73]{Aliprantis:03}, $Y=X^u$ (even when $\mu$ is
  only assumed to be $\sigma$-finite).
\end{example}

\begin{example}
  In Section~\ref{sec:BFS}, we considered the case when $X$ is a an
  order complete Banach function space over a $\sigma$-finite measure
  $\mu$ and $Y=L_0(\mu)$.  Since $X$ is an order dense ideal in
  $L_0(\mu)$ and the latter is universally complete, we have
  $X^u=L_0(\mu)$.
\end{example}

\begin{remark}\label{un-gen}
  General un-convergence often reduces to un-convergence on $X^u$ as
  follows. Suppose that $X$ is an order complete Banach lattice which
  is an order dense ideal in a vector lattice~$Y$. We are interested
  in the un-convergence on $Y$ induced by~$X$. Note that $X$ is an
  order dense ideal in $X^u$ and that $X^u=Y^u$, so we may assume that
  $X\subseteq Y\subseteq X^u$. It follows that the un-convergence on
  $Y$ is the restriction of un-convergence on~$X^u$.
\end{remark}

\begin{remark}\label{noncompl}
  We now extend our definition of un-topology on $X^u$ to the case
  when $X$ is not order complete (and, therefore, need not be an ideal
  in $X^u$). In this case, we consider the un-topology on $X^u$
  induced by~$X^\delta$. In this setting, Remark~\ref{un-gen} remain
  valid provided that $Y$ is order complete. Since $X$ is majorizing
  in~$X^\delta$, the ``native'' un-topology of $X$ still agrees with
  the restriction to $X$ of the un-topology on $X^u$ induced
  by~$X^\delta$.
\end{remark}
 
In \cite[Proposition~6.2]{KMT}, it was shown that an order continuous
Banach lattice is un-complete iff it is finite-dimensional. We will
show next that the situation is completely different when we consider
un-topology on $X^u$ instead of~$X$.

\begin{theorem}
  If $X$ is an order continuous Banach lattice then the un-topology on
  $X^u$ is complete.
\end{theorem}

\begin{proof}
  The proof is similar to that of \cite[Theorem~6.4]{KMT}.
  Suppose first that $X$ has a weak unit. Then $X$ may be identified
  with an order dense ideal of $L_1(\mu)$ for some finite measure
  $\mu$; see \cite[Theorem~1.b.14]{Lindenstrauss:79} or
  \cite[Section~4]{GTX}. It follows that
  $X^u=L_1(\mu)^u=L_0(\mu)$. By Corollary~\ref{bfs-metriz}, the
  un-topology on $L_0(\mu)$ induced by $X$ is metrizable. Therefore,
  it suffices to show that it is sequentially
  complete. Theorem~\ref{bfs-measure} yields that this topology is the
  topology of convergence in measure, which is sequentially complete
  by \cite[Theorem~2.30]{Folland:99}.

  Now we consider the general case. Let $(y_\alpha)$ be a un-Cauchy
  net in~$X^u$. Without loss of generality, $y_\alpha\ge 0$; otherwise
  we separately consider the nets $(y_\alpha^+)$ and
  $(y_\alpha^-)$. By \cite[Proposition~1.a.9]{Lindenstrauss:79} (see
  also Section~4 in~\cite{KMT}), $X$ can be written as the closure of
  a direct sum of a family $\mathcal B$ of pair-wise disjoint
  principal bands.

  Fix $B\in\mathcal B$; let $\widetilde{B}$ be the band of $X^u$
  generated by~$B$. Being a band in~$X^u$, $\widetilde{B}$ is
  universally complete. Since $B$ is order dense in~$\widetilde{B}$,
  we may identify $\widetilde{B}$ with~$B^u$. Let
  $P_{\widetilde{B}}\colon X^u\to\widetilde{B}$ be the band projection
  for~$\widetilde{B}$. It is easy to see that the net
  $(P_{\widetilde{B}}y_\alpha)$ is un-Cauchy in $\widetilde B$ with
  respect to~$B$. Since $B$ has a weak unit, the first part of the
  proof yields that there exists $y_B\in\widetilde B$
  such that $P_{\widetilde{B}}y_\alpha$ un-converges to $y_B$ in
  $\widetilde B$ (and, therefore, in $X^u$) with respect to~$B$. It
  follows from Proposition~\ref{ops} that $y_B\ge 0$.

  Put $y=\sup\{y_B\mid B\in\mathcal B\}$. This supremum exists in
  $X^u$ because $X^u$ is universally complete. We claim that
  $P_{\widetilde{B}}y=y_B$ for each $B\in\mathcal B$. Indeed, let
  $\gamma$ be a finite subset of~$\mathcal B$. Define
  $z_\gamma=\bigvee_{B\in\gamma}y_B$. Then $(z_\gamma)$ may be viewed
  as a net, and $z_\gamma\uparrow y$. Let $B\in\mathcal B$. For every
  $\gamma$ such that $B\in\gamma$, the order continuity of
  $P_{\widetilde{B}}$ yields
  \begin{math}
    y_B=P_{\widetilde{B}}z_\gamma\uparrow P_{\widetilde{B}}y,
  \end{math}
  so that $P_{\widetilde{B}}y=y_B$ and, therefore,
  $P_{\widetilde{B}}y_\alpha\goesunB P_{\widetilde{B}}y$ for every~$B$.

  It is left to show that $y_\alpha\goesunX y$ in~$X^u$. The argument
  is similar to the proof of \cite[Theorem~4.12]{KMT} and we leave it
  as an exercise.
\end{proof}

\section{Un-compact intervals}

Let $X$ be a Banach lattice. It is well known that order intervals in
$X$ are norm compact iff $X$ is atomic and order continuous; see,
e.g., Theorem~6.2 in~\cite{Wnuk:99}. Since un-convergence on order
intervals of $X$ agrees with norm convergence, we immediately
conclude that order intervals in $X$ are un-compact iff $X$ is atomic
and order continuous. Recall that $[a,b]=a+[0,b-a]$ whenever $a\le b$;
therefore, when dealing with order intervals, it often suffices to
consider order intervals of the form $[0,u]$.

Suppose now that $X$ is an order dense ideal in a vector lattice $Y$
and consider order intervals in~$Y$. It is easy to see that they are
un-closed. Indeed, suppose that $y_\alpha\goesun y$ in $Y$ and
$u\in Y_+$ such that $0\le y_\alpha\le u$ for every $\alpha$; it
follows from Proposition~\ref{ops} that $0\le y\le u$.

\begin{proposition}\label{int-Xu}
  Let $X$ be a Banach lattice. TFAE:
  \begin{enumerate}
  \item\label{int-Xu-atom} $X$ is atomic and order continuous;
  \item\label{int-Xu-Xu} Order intervals in $X^u$ are un-compact.
  \end{enumerate}
\end{proposition}

\begin{proof}
  \eqref{int-Xu-Xu}$\Rightarrow$\eqref{int-Xu-atom} Suppose that order
  intervals in $X^u$ are un-compact. Since $X^\delta$ is an ideal
  in~$X^u$, every order interval in $X^\delta$ is an order interval
  in~$X^u$.  Since the un-topology on $X^\delta$ is the restriction to
  $X^\delta$ of the un-topology on~$X^u$, order intervals in
  $X^\delta$ are un-compact and, therefore, norm compact. Furthermore,
  since $X$ is a closed sublattice of~$X^\delta$, order intervals of
  $X$ are closed subsets of order intervals of $X^\delta$ and,
  therefore, are compact in~$X$. Therefore, $X$ has compact intervals;
  it follows that $X$ is atomic and order continuous.

  \eqref{int-Xu-atom}$\Rightarrow$\eqref{int-Xu-Xu} As in
  Proposition~\ref{atom-uo-ptwise} and Example~\ref{Xu-atom}, we may
  assume that $X^u=\mathbb R^A$ for some set $A$ and the
  un-convergence on $\mathbb R^A$ agrees with the point-wise
  convergence. It is left to observe that order intervals in $\mathbb
  R^A$ are compact with respect to the topology of point-wise
  convergence. Indeed, given $u\in\mathbb R^A_+$, we have
  $[0,u]=\prod_{a\in A}\bigl[0,u(a)\bigr]$, Since each of the
  intervals $\bigl[0,u(a)\bigr]$ is compact in~$\mathbb R$, the set
  $[0,u]$ is compact by Tychonoff's Theorem.
\end{proof}

\begin{theorem}
  Let $X$ be a Banach lattice such that $X$ is an order dense ideal in a
  vector lattice~$Y$. Order intervals in $Y$ are un-compact iff $X$ is
  atomic and order continuous and $Y$ is order complete.
\end{theorem}

\begin{proof}
  Suppose that $X$ is atomic and order continuous and $Y$ is order
  complete. By Proposition~\ref{int-Xu}, order intervals of $X^u$ are
  un-compact. Since $X$ is order dense in~$Y$, we have $Y^u=X^u$, so
  we may assume that $Y$ is a sublattice of~$X^u$. Since $Y$ is order
  complete, $Y$ is an ideal of~$X^u$. Therefore, order intervals of $Y$
  are also order intervals in $X^u$ and the un-topology on $Y$ is the
  restriction of the un-topology on $X^u$ to~$Y$. It follows that
  order intervals of $Y$ are un-compact.

  Conversely, suppose that order intervals of $Y$ are un-compact. It
  follows, in particular, that order intervals of $X$ are un-compact,
  hence norm compact and, therefore, $X$ is atomic and order
  continuous. To show that $Y$ is order complete, suppose that
  $0\le y_\alpha\uparrow\le u$ in~$Y$. Since $[0,u]$ is un-compact,
  there is a subnet $(z_\gamma)$ of $(y_\alpha)$ such that
  $z_\gamma\goesun z$ for some $z\in Y$. Since order intervals are
  un-closed, we have $z\in[0,u]$. Since the net $(y_\alpha)$ is
  increasing, so is $(z_\gamma)$ and, therefore, $z_\gamma\uparrow z$
  by Proposition~\ref{un-monot}. It follows that $y_\alpha\uparrow z$.
\end{proof}

\section{Spaces with a strong unit}

In this section, we consider the case when $X$ is a Banach lattice
with a strong unit, such that $X$ is an order dense ideal in a vector
lattice~$Y$. Cf.\ Example~\ref{linfty}.

We start by recalling some preliminaries.  Let $E$ be a vector lattice
and $e\in E_+$. For $x\in E$, we define
\begin{displaymath}
  \norm{x}_e=\inf\bigl\{\lambda>0\mid \abs{x}\le\lambda e\bigr\}.
\end{displaymath}
This expression is finite iff $x$ belongs to the principal ideal
$I_e$; it defines a lattice norm on~$I_e$. For a net $(y_\alpha)$ and
a vector $y$ in~$Y$, we say that $y_\alpha$ \term{converges to} $y$
\term{uniformly with respect to} $e$ if $\norm{y_\alpha-y}_e\to
0$. Note that this implies that $y_\alpha-y\in I_e$ for all
sufficiently large~$\alpha$. $E$ is said to
be \term{uniformly complete} if $\bigl(I_e,\norm{\cdot}_e\bigr)$ is
complete for every $e\in E_+$. Every $\sigma$-order
complete vector lattice and every Banach lattice is uniformly
complete; see \cite[\S 42]{Luxemburg:71} and \cite[Theorem~4.21]{Aliprantis:06}.
If $\bigl(I_e,\norm{\cdot}_e\bigr)$ is complete then it is
lattice isometric to $C(K)$ for some compact Hausdorff space $K$ with
$e$ corresponding to the constant one function~$\one$. In particular,
every Banach lattice with a strong unit $e$ is lattice isomorphic to
$C(K)$ with $e$ corresponding to~$\one$. We refer the reader to
Section~3.1 in~\cite{Abramovich:02} for further details. It was
observed in~\cite{KMT} that if $X$ is a Banach lattice with a strong
unit $e$ then un-convergence in $X$ agrees with norm convergence,
which, in turn, agrees with the uniform convergence with respect to~$e$.

\begin{proposition}\label{str-u}
  Let $X$ be a Banach lattice with a strong unit~$e$, such that $X$ is
  an order dense ideal in a vector lattice~$Y$. For a net $(y_\alpha)$
  in~$Y$, $y_\alpha\goesun 0$ iff $(y_\alpha)$ converges to zero
  uniformly with respect to $e$ (in particular, a tail of $(y_\alpha)$
  is contained in $X$).
\end{proposition}

\begin{proof}
  First, we consider the special case when $Y$ is uniformly
  complete. Note that $X$ equals the principal ideal generated by $e$
  in~$Y$. It follows that the original norm of $X$ is equivalent to
  $\norm{\cdot}_e$. By Proposition~\ref{qip}, $y_\alpha\goesunX 0$ iff
  $\norm{\abs{y_\alpha}\wedge e}_X\to 0$, which is equivalent to
  $\norm{\abs{y_\alpha}\wedge e}_e\to 0$. Note that if
  $\norm{\abs{y_\alpha}\wedge e}_e<1$ then $\abs{y_\alpha}\le e$, and,
  therefore, $y_\alpha\in X$ and
  $\norm{\abs{y_\alpha}\wedge e}_e=\norm{y_\alpha}_e$. It is now easy
  to see that $\norm{\abs{y_\alpha}\wedge e}_e\to 0$ iff $y_\alpha$
  converges to $0$ uniformly with respect to~$e$.

  Now we consider the general case. By Proposition~\ref{deltas},
  $y_\alpha\goesunX 0$ in $Y$ iff $y_\alpha\goesunXd 0$ in
  $Y^\delta$. Observe that $Y^\delta$ is order complete, hence
  uniformly complete. Furthermore, $e$ is a strong unit
  in~$X^\delta$. By the special case, $y_\alpha\goesunXd 0$ in
  $Y^\delta$ iff a tail of $(y_\alpha)$ is contained in $X^\delta$ and
  converges to zero uniformly with respect to~$e$. It follows that
  $\abs{y_\alpha}\le e$ for all sufficiently large~$\alpha$. Since $X$
  is an ideal in~$Y$, we conclude that $y_\alpha\in X$.
\end{proof}

Suppose that $Y$ is a uniformly complete vector lattice and $e\in Y_+$
is a weak unit. Put $X=\bigl(I_e,\norm{\cdot}_e\bigr)$. Then $X$ is a
Banach lattice with a strong unit, and an order dense ideal in
$Y$. Consider the un-topology induced by $X$ on~$Y$. The
following result is an immediate corollary of Proposition~\ref{str-u}.

\begin{corollary}
  Let $Y$ be a uniformly complete vector lattice, $e\in Y_+$ a weak
  unit, and $X=\bigl(I_e,\norm{\cdot}_e\bigr)$. For a net $(y_\alpha)$
  in~$Y$, $y_\alpha\goesunX 0$ iff $y_\alpha$ converges to zero uniformly with
  respect to~$e$.
\end{corollary}

Let $X$ be a Banach lattice with a strong unit~$e$. As in
Remark~\ref{noncompl}, we consider the un-topology on $X^u$ induced
by~$X^\delta$. Note that $X^u=(X^\delta)^u$. By
Proposition~\ref{str-u}, for a net $(y_\alpha)$ in~$X^u$,
$y_\alpha\goesun 0$ iff a tail of $(y_\alpha)$ is contained in
$X^\delta$ and converges to zero uniformly with respect to~$e$.

Since $X$ is a Banach lattice with a strong unit, up to a lattice
isomorphism we may identify $X$ with $C(K)$ for some compact
space~$K$. Furthermore, $X^\delta$ is an order complete Banach lattice
with strong unit, so we can identify it with $C(Q)$ for some
extremally disconnected compact space~$Q$. Since $X^u=(X^\delta)^u$,
by \cite[Theorem~7.29]{Aliprantis:03}, we can identify $X^u$ with
$C_\infty(Q)$. Therefore, for a net $(f_\alpha)$ in $C_\infty(Q)$,
$f_\alpha\goesun 0$ iff a tail of $(f_\alpha)$ is contained in $C(Q)$
and converges to zero uniformly on~$Q$.

\section{Un-convergence versus uo-convergence}

Let $Y$ be a vector lattice. Recall that a net $(y_\alpha)$ in $Y$ is
said to \term{converge in order} to $y$ if there exists a net
$(z_\gamma)$ in~$Y$, which may, generally, have a different index set,
such that $z_\gamma\downarrow 0$ and
\begin{math}
  \forall\gamma\ \exists\alpha_0\ \forall\alpha\ge\alpha_0\
  \abs{y_\alpha-y}\le z_\gamma.
\end{math}
In this case, we write $y_\alpha\goeso y$.  We say that $y_\alpha$
\term{uo-converges} to $y$ and write $y_\alpha\goesuo y$ if
$\abs{y_\alpha-y}\wedge u\goeso 0$ for every $u\in Y_+$. We refer the
reader to~\cite{GTX} for a review of order and uo-convergence. We will
only mention three facts here. First, it was observed in
\cite[Corollary~3.6]{GTX} that every disjoint sequence is
uo-null. Second, for sequences in $L_0(\mu)$, uo-convergence coincides
with almost everywhere (a.e.) convergence. For the third fact, we need
the concept of a regular sublattice.  Recall that a sublattice $Z$ of
$Y$ is \term{regular} if $z_\alpha\downarrow 0$ in $Z$ implies
$z_\alpha\downarrow 0$ in~$Y$. In this case, Theorem~3.2 of~\cite{GTX}
asserts that $z_\alpha\goesuo 0$ in $Z$ iff $z_\alpha\goesuo 0$ in $Y$
for every net $(z_\alpha)$ in~$Z$.  In particular, if $Y$ is a regular
sublattice of $L_0(\mu)$ then uo-convergence coincides with a.e.\
convergence for sequences in~$Y$. Therefore, uo-convergence may be
viewed as a generalization of a.e.\ convergence.

Suppose, as before, that $X$ is a normed lattice which is an ideal
in~$Y$. Our goal is to compare the un-convergence on $Y$ induced by
$X$ with the uo-convergence on~$Y$.

The following is an extension of Proposition~3.5 in~\cite{KMT}.

\begin{proposition}\label{disj}
  The following are equivalent.
  \begin{enumerate}
  \item\label{disj-oc} $X$ is order continuous;
  \item\label{disj-seq} Every disjoint sequence in $Y$ is un-null;
  \item\label{disj-net} Every disjoint net in $Y$ is un-null. 
  \end{enumerate}
\end{proposition}

\begin{proof}
  To prove \eqref{disj-oc}$\Rightarrow$\eqref{disj-seq}, observe that
  if $(y_n)$ is a disjoint sequence in $Y$ then it is uo-null in~$Y$.
  In particular, for every $x\in X_+$,
  the sequence $\abs{y_n}\wedge x$ converges to zero in order and,
  therefore, in norm. The proof that
  \eqref{disj-oc}$\Leftarrow$\eqref{disj-seq}$\Leftrightarrow$\eqref{disj-net}
  is straightforward, cf.\ \cite[Proposition~3.5]{KMT}.
\end{proof}

\begin{proposition}\label{uo-un}
  The following are equivalent.
  \begin{enumerate}
  \item\label{uo-un-ocn} $X$ is order continuous;
  \item\label{uo-un-X} $x_\alpha\goesuo 0$ in $X$ implies
    $x_\alpha\goesun 0$ in $X$ for every net $(x_\alpha)$ in~$X$;
  \item\label{uo-un-Y} $y_\alpha\goesuo 0$ in $Y$ implies
    $y_\alpha\goesun 0$ in $Y$ for every net $(y_\alpha)$ in~$Y$.
  \end{enumerate}
\end{proposition}

\begin{proof}
  Note that being an ideal in~$Y$, $X$ is a regular sublattice and,
  therefore, $x_\alpha\goesuo 0$ in $X$
  iff $x_\alpha\goesuo 0$ in $Y$ for every net $(x_\alpha)$ in~$X$.

  \eqref{uo-un-ocn}$\Rightarrow$\eqref{uo-un-Y} Suppose that
  $y_\alpha\goesuo 0$ in~$Y$. Fix $x\in X_+$. Then $\abs{y_\alpha}\wedge
  x\goeso 0$ in~$X$, hence $\abs{y_\alpha}\wedge x\goesnorm 0$ and,
  therefore, $y_\alpha\goesun 0$.

  \eqref{uo-un-Y}$\Rightarrow$\eqref{uo-un-X} Suppose $x_\alpha\goesuo
  0$ in~$X$. Then $x_\alpha\goesuo 0$ in $Y$ and, therefore,
  $x_\alpha\goesun 0$ in $Y$ and in~$X$.

  \eqref{uo-un-X}$\Rightarrow$\eqref{uo-un-ocn} We will apply
  Proposition~\ref{disj} with $X=Y$. Every disjoint sequence in $X$ is
  uo-null, hence it is un-null by the assumption.
\end{proof}

Recall the following standard fact; see, e.g., \cite[Theorem~2.30]{Folland:99}.

\begin{theorem}\label{Fol}
  Let $\mu$ be a finite measure. Every sequence in $L_0(\mu)$ which
  converges in measure has a subsequence which converges a.e.
\end{theorem}

It is a natural question whether this result can be generalized with
a.e.\ convergence and convergence in measure replaced with uo- and
un-convergences, respectively. A partial advance in this direction
was made in Proposition~4.1 of~\cite{DOT}:

\begin{proposition}\cite{DOT}
\label{un-uo-BL}
  If $X$ is a Banach lattice and $x_n\goesun 0$ in $X$ then
  $x_{n_k}\goesuo 0$ for some subsequence $(x_{n_k})$.
\end{proposition}

However, formally speaking, Proposition~\ref{un-uo-BL} is not a
generalization of Theorem~\ref{Fol} because $L_0(\mu)$ is not a Banach
lattice. Using the framework of this paper, we are now ready to
produce an appropriate extension of Theorem~\ref{Fol}.

\begin{theorem}\label{un-uo-wu}
  Let $X$ be an order continuous Banach lattice with a weak unit, such
  that $X$ is an order dense ideal in a vector lattice $Y$; let
  $(y_n)$ be a sequence in $Y$ such that $y_n\goesunX 0$. Then there is
  a subsequence $(y_{n_k})$ such that $y_{n_k}\goesuo 0$ in~$Y$.
\end{theorem}

\begin{proof}
  \emph{Special case:} $Y=X^u$. Represent $X$ as an order dense ideal
  in $L_1(\mu)$ for some finite measure~$\mu$. We may then identify
  $X^u$ with $L_0(\mu)$. By Theorem~\ref{bfs-measure}, $y_n\goesunX 0$
  yields $y_n\goesmu 0$ in $L_0(\mu)$. By Theorem~\ref{Fol}, there
  exists a subsequence $(y_{n_k})$ such that $y_{n_k}\goesae 0$ and,
  therefore, $y_{n_k}\goesuo 0$ in $L_0(\mu)$.

  \emph{General case.} Since $X$ is order dense in~$Y$, we have
  $X^u=Y^u$. Therefore, we may assume without loss of generality, that
  $Y$ is an order dense sublattice of~$X^u$. Since $y_n\goesunX 0$
  in~$Y$, it is trivial that $y_n\goesunX 0$ in~$X^u$. The special
  case yields $y_{n_k}\goesuo 0$ in $X^u$ for some subsequence
  $(y_{n_k})$. Since $Y$ is an order dense sublattice of~$X^u$, it is
  regular by \cite[Theorem~1.23]{Aliprantis:03}. It follows that
  $y_{n_k}\goesuo 0$ in~$Y$.
\end{proof}

Note that in Proposition~\ref{un-uo-BL}, $X$ need not have a weak
unit. Therefore, it is natural to ask whether a weak unit is really
needed in Theorem~\ref{un-uo-wu}. The following example shows that the
weak unit assumption in  Theorem~\ref{un-uo-wu} cannot be removed.

\begin{example}
  We are going to construct an order continuous Banach lattice $X$
  with no weak units and a vector lattice $Y$ such that $X$ is an
  order dense ideal in $Y$ (actually, $Y=X^u$) and a sequence $(y_n)$
  in $Y$ such that $y_n\goesunX 0$ in $Y$ and yet no subsequence of
  $(y_n)$ is uo-null.

  Let $\Gamma$ be an infinite set (we will choose a specific $\Gamma$
  later).  Let $X$ be the $\ell_1$-sum of infinitely many copies of
  $L_1[0,1]$ indexed by~$\Gamma$. That is, $X$ is the space of
  functions $x\colon\Gamma\to L_1[0,1]$ such that
  \begin{math}
    \norm{x}:=\sum_{\gamma\in\Gamma}\norm{x(\gamma)}_{L_1[0,1]}
  \end{math}
  is finite. We may also view $x$ as a function on the union of
  $\abs{\Gamma}$ many copies of $[0,1]$ or, equivalently, on
  $[0,1]\times\Gamma$. We write $x^\gamma$ instead of $x(\gamma)$ and
  call it a component of~$x$. It is easy to see that $X$ is an
  AL-space; in particular, it is order continuous. It is also easy to
  see that each $x\in X$ has at most countably many non-zero
  components. Let $Z$ be the subset of $X$ consisting of those $x$
  which have finitely many non-zero components. It can be easily
  verified that $Z$ is a norm dense ideal in~$X$.

  Let $Y$ be the direct sum of infinitely many copies of $L_0[0,1]$
  indexed by~$\Gamma$. That is, $Y=\bigl(L_0[0,1]\bigr)^\Gamma$, the space of
  all functions from $\Gamma$ to $L_0[0,1]$. Again, we write
  $y^\gamma$ instead of $y(\gamma)$ and may view $y$ as a real-valued
  function on $[0,1]\times\Gamma$. It is easy to see that $Y$ is a
  vector lattice under component-wise lattice operations, and $X$ is
  an order dense ideal in~$Y$. Furthermore, $Y=X^u$. We equip $Y$ with
  the un-topology induced by~$X$.

  Let $(y_n)$ be an arbitrary sequence in~$Y_+$. We claim that in
  order for it to be un-null, it suffices that the sequence
  $y_n^\gamma$ converges to zero in measure in for
  every~$\gamma$. Indeed, the latter implies that
  $\norm{y_n^\gamma\wedge u}_{L_1[0,1]}\to 0$ for every positive
  $u\in L_1[0,1]$. It follows easily that $\norm{y_n\wedge x}_X\to 0$
  for every $x\in Z_+$. Proposition~\ref{dense} now yields that
  $y_n\goesunX 0$.

  Furthermore, it can be easily verified that if a sequence $(y_n)$ in
  $Y$ is uo-null then for every $\gamma\in\Gamma$ one has
  $y_n^\gamma\goesuo 0$ in $L_0[0,1]$, hence $y_n^\gamma\goesae 0$.

  We now specify $\Gamma$: let it be the set of all strictly
  increasing sequences of natural numbers. Fix a sequence $(f_k)$ in
  $L_0[0,1]$ such that $f_k$ converges to zero in measure but not
  a.e. For each~$n$, define $y_n$ in $Y$ as follows: for each
  $\gamma\in\Gamma$, let $\gamma=(n_k)$ and put $y_n^\gamma=f_k$ if
  $n=n_k$ and $y_n^\gamma=0$ if $n$ is not in~$\gamma$. It follows
  that $(y_{n_k}^\gamma)$ as a sequence in $k$ is exactly $(f_k)$,
  while the rest of the terms of $(y_n^\gamma)$ are zeros. Therefore,
  the sequence $(y_n^\gamma)$ converges in measure to zero for
  every~$\gamma$, hence $y_n\goesun 0$. On the other hand, $(y_n)$ has
  no uo-null subsequences. Indeed, suppose that $y_{n_k}\goesuo 0$ for
  some subsequence $(n_k)=\gamma$. Then $y^\gamma_{n_k}\goesae 0$ in
  $L_0[0,1]$, so that $f_k\goesae 0$; a contradiction.
\end{example}

\subsection*{Acknowledgement and further remarks.}
We would like to thank Niushan Gao and Mitchell Taylor for valuable
discussions. After this project was completed, M.~Taylor has
generalized some of the results to the setting of locally solid vector
lattices; see~\cite{Taylor}.

\end{document}